\documentclass[12pt,fleqn,leqno]{amsart}
\usepackage{amsfonts,amssymb}
\usepackage{mathrsfs}
\usepackage{enumitem}

\usepackage[svgnames]{xcolor}
\usepackage[colorlinks,linkcolor=blue,citecolor=Green]{hyperref}
\usepackage{titlesec}
\usepackage[a4paper,text={14cm,20.5cm},centering]{geometry}
\usepackage{tikz-cd}
\usetikzlibrary{decorations.pathmorphing}
\titleformat{\section}[block]
 {\bfseries}
 {\thesection.}
 {\fontdimen2\font}
 {}

\setlist{noitemsep}
\setenumerate{labelindent=\parindent,label=\upshape{(\alph*)}}

\newtheorem{theorem}{Theorem}[section]
\newtheorem{corollary}[theorem]{Corollary}
\newtheorem{proposition}[theorem]{Proposition}

\DeclareMathOperator{\uhr}{\upharpoonright}
\DeclareMathOperator\sto{\leadsto}

\renewcommand{\emptyset}{\varnothing}
\numberwithin{equation}{section}

\begin{document}

\author{Valentin Gutev}

\address{Institute of Mathematics and Informatics, Bulgarian Academy
  of Sciences, Acad. G. Bonchev Street, Block 8, 1113 Sofia, Bulgaria}

\email{\href{mailto:gutev@math.bas.bg}{gutev@math.bas.bg}}



\title{Two Selection Theorems for Extremally Disconnected Spaces}

\begin{abstract}
  The paper contains a very simple proof of the classical Hasumi's
  theorem that each usco mapping defined on an extremally disconnected
  space has a continuous selection. The paper also contains a very
  simple proof of a recent result about extension of densely defined
  continuous selections for compact-valued continuous mappings, in
  fact a generalisation of this result to all usco mappings with a
  regular range.
\end{abstract}

\dedicatory{Dedicated to Professor Georgi Dimov on the occasion of his
  75th birthday}

\date{\today}
\maketitle

\section{Introduction}

All spaces in this paper are Hausdorff topological spaces. For spaces
(sets) $X$ and $Y$, we write $\varphi:X\sto Y$ to designate that
$\varphi$ is a map from $X$ to the \emph{nonempty} subsets of $Y$. In
Michael's selection theory, such a map is commonly called a
\emph{set-valued mapping} (also a \emph{multifunction}, or simply a
\emph{carrier} \cite{michael:56a}). In this regard, let us recall that
a usual map $f:X\to Y$ is a \emph{selection} (or a \emph{single-valued
  selection}) for $\varphi:X\sto Y$ if $f(x)\in\varphi(x)$ for every
$x\in X$.\medskip

For spaces $X$ and $Y$, it is customary to say that a mapping
$\varphi:X\sto Y$ is \emph{compact-valued} (\emph{singleton-valued},
etc.) if every set $\varphi(x)\subset Y$, $x\in X$, is compact
(respectively, singleton, etc.). Evidently, each singleton-valued
mapping $\varphi:X\sto Y$ is identical to a single-valued map from $X$
to $Y$. A mapping $\varphi:X\sto Y$ is \emph{upper semi-continuous},
or \emph{u.s.c.}, if the preimage
\[
  \varphi^{-1}[F]=\{x\in X: \varphi(x)\cap F\neq \emptyset\}
\]
is closed in $X$ for every closed $F\subset Y$; equivalently, it is
u.s.c.\ if the set
\[
\varphi^{\#}[U]=X\setminus \varphi^{-1}[Y\setminus U]=\{x\in X:
\varphi(x)\subset U\}
\]
is open in $X$ for every open $U\subset Y$. For convenience, we say
that $\varphi:X\sto Y$ is \emph{usco} if it is u.s.c.\ and
{compact-valued}. Finally, let us recall that $\varphi:X\sto Y$ is
\emph{lower semi-continuous}, or \emph{l.s.c.}, $\varphi^{-1}[U]$ is
open in $X$ for every open $U\subset Y$, and we say that $\varphi$ is
\emph{continuous} if it is both l.s.c.\ and u.s.c.\medskip

A continuous map $f:X\to Y$ is \emph{perfect} if it is closed and each
$f^{-1}(y)$, $y\in Y$, is a compact subset of $X$. Similarly, we will
say that an usco mapping $\varphi:X\sto Y$ is \emph{perfect} if
$\varphi[S]=\bigcup_{x\in S}\varphi(x)$ is closed in $Y$ for every
closed $S\subset X$, and each $\varphi^{-1}(y)=\varphi^{-1}[\{y\}]$,
$y\in Y$, is compact. Evidently, perfect singleton-valued mappings are
identical to perfect single-valued maps.  Moreover, it follows from
\cite[Proposition 1.1]{choban:70a} that $\varphi:X\sto Y$ is perfect
if and only if the projections $\pi_X:\Gamma(\varphi)\to X$ and
$\pi_Y:\Gamma(\varphi)\to Y$ from the graph
$\Gamma(\varphi)\subset X\times Y$ of $\varphi$ are perfect maps:
\begin{center}
\begin{tikzcd}[column sep=small]
  & \Gamma(\varphi) \arrow{dl}[swap]{\pi_X}  \arrow[dr, "\pi_Y"] & \\
  X \arrow[rightsquigarrow, "\varphi"]{rr} & & Y
\end{tikzcd}
\end{center}

Finally, let us recall that a space $X$ is \emph{extremally
  disconnected} if $\overline{U}$ is open for every open $U\subset
X$. The following interesting selection theorem was obtained by Hasumi
in \cite[Theorem 1.1]{MR0248773}.

\begin{theorem}
  \label{theorem-min-usco-v3:1}
  Let $X$ be an extremally disconnected space, $Y$ be a regular space
  and $\varphi:X\sto Y$ be an usco mapping. Then $\varphi$ has a
  continuous selection. If moreover $\varphi$ is also perfect, then it
  has a perfect selection.
\end{theorem}

Recently, the following related result was obtained in \cite[Theorem
1]{Pimienta}.

\begin{theorem}
  \label{theorem-min-usco-v4:1}
  Let $X$ be an extremally disconnected regular space, $Y$ be a
  metrizable space and $\varphi:X\sto Y$ be a compact-valued
  continuous mapping. Then for each dense subset $A\subset X$, each
  continuous selection for $\varphi\uhr A$ can be extended to a
  continuous selection for~$\varphi$.
\end{theorem}

The proof of Theorem \ref{theorem-min-usco-v3:1} in \cite{MR0248773}
is extended on several pages and is technically demanding. A simple
proof of this theorem for an arbitrary range $Y$ was given by Shapiro
in \cite[Corollary 1]{Shapiro1976}. However, the proof in
\cite{Shapiro1976} is not direct being based on the existence of an
extremally disconnected space $E$ and a perfect onto map $\pi:E\to
Y$. The proof of Theorem \ref{theorem-min-usco-v4:1} is also
technically demanding being based on convergence of nets.  In this
paper, we will give a very simple direct proof of Theorem
\ref{theorem-min-usco-v3:1}, see Propositions
\ref{proposition-min-usco-v9:1} and
\ref{proposition-min-usco-v10:1}. Section~\ref{sec:minim-relat-extr}
also contains some applications and related results. The last Section
\ref{sec:extens-select-extr} contains a very simple proof that Theorem
\ref{theorem-min-usco-v4:1} is valid for any usco mapping with a
completely regular range, see Proposition
\ref{proposition-min-usco-v16:1}. Using projective spaces, we will
also show that this theorem is valid for any regular range $Y$ as
well.

\section{Minimal Mappings and Extremal Disconnectedness}
\label{sec:minim-relat-extr}

The considerations in this section are based on the following very
simple observation which is behind the selection property in Theorem
\ref{theorem-min-usco-v3:1}. 

\begin{proposition}
  \label{proposition-min-usco-v9:1}
  Let $\psi:X\sto Y$ be a mapping such
  that 
  \begin{equation}
    \label{eq:min-usco-v9:1}
    \psi^{-1}[U]\subset 
    \psi^\#\left[\overline{U} \right]\quad \text{for every open $U\subset
      Y$.}
  \end{equation}
  Then $\psi$ is singleton-valued, i.e.\ a usual map from $X$ to $Y$.
\end{proposition}

\begin{proof}
  Take a point $p\in X$ and an open set $U\subset Y$ with $\psi(p)\cap
  U\neq \emptyset$. Then $p\in \psi^{-1}[U]$ and it follows from
  \eqref{eq:min-usco-v9:1} that $\psi(p)\subset\overline{U}$. Since 
  $Y$ is Hausdorff, this implies that $\psi(p)$ is a singleton. 
\end{proof}

The condition in \eqref{eq:min-usco-v9:1} is naturally related to
minimal usco mappings. Let us recall that an usco (perfect) mapping
$\psi:X\sto Y$ is \emph{minimal} if $\psi=\eta$ for any usco (perfect)
mapping $\eta:X\sto Y$ whose graph $\Gamma(\eta)$ is contained in the
graph $\Gamma(\psi)$ of $\psi$.  It is a well-known folklore result
that if $\psi:X\sto Y$ is a minimal usco mapping and $U\subset Y$ is
an open set, then $\overline{\psi^{-1}[U]}=\overline{\psi^\#[U]}$ and
$\psi(p)\subset \overline{U}$ for each interior point
$p\in \overline{\psi^\#[U]}$, see e.g.\ \cite[Lemma
2]{Valov1987}. Evidently, in case of an extremally disconnected space
$X$, we have the following relaxed form of this property.

\begin{proposition}
  \label{proposition-min-usco-v10:1}
  If $X$ is an extremally disconnected space and $\psi:X\sto Y$ is a
  minimal usco (perfect) mapping, then 
  \begin{equation}
    \label{eq:min-usco-v2:1}
    \overline{\psi^{-1}[U]}=\overline{\psi^\#[U]}\subset 
    \psi^\#\left[\overline{U}\right]\quad \text{for every open $U\subset
      Y$.}
  \end{equation}
  In particular, $\psi$ satisfies \eqref{eq:min-usco-v9:1} and is
  therefore singleton-valued.
\end{proposition}

\begin{proof}
  Here is a brief outline of the proof. Take an open set $U\subset Y$
  and consider the clopen sets $G=\overline{\psi^\#[U]}$ and
  $H=X\setminus G$. Since
  $G\subset \psi^{-1}\left[\overline{U}\right]$ and
  $H\subset \psi^{-1}[Y\setminus U]$, we can define an usco mapping
  $\eta:X\sto Y$ by $\eta(x)=\psi(x)\cap\overline{U}$ if $x\in G$ and
  $\eta(x)=\psi(x)\setminus U$ if $x\in H$. It is also evident that
  $\eta$ is perfect whenever so is $\psi$.  Accordingly, $\psi=\eta$
  because $\Gamma(\eta)\subset \Gamma(\psi)$. Thus,
  \eqref{eq:min-usco-v2:1}~holds.
\end{proof}

It is well known and easily follows from the Kuratowski-Zorn lemma
that every usco mapping contains a minimal usco one. The existence of
minimal perfect mappings was shown in \cite[Lemma 2.1]{MR0248773}. In
fact, as remarked in the Introduction, a mapping $\varphi:X\sto Y$ is
perfect precisely when the projections ${\pi_X:\Gamma(\varphi)\to X}$
and $\pi_Y:\Gamma(\varphi)\to Y$ are perfect maps. Hence, if
$\varphi:X\sto Y$ is a perfect mapping and ${\psi:X\sto Y}$ is a
minimal usco mapping with ${\Gamma(\psi)\subset \Gamma(\varphi)}$,
then $\psi$ is itself perfect.  Thus, since each singleton-valued
u.s.c.\ mapping is identical to a continuous single-valued map,
Proposition \ref{proposition-min-usco-v10:1} immediately implies
Shapiro's generalisation \cite[Corollary 1]{Shapiro1976} of Hasumi's
Theorem~\ref{theorem-min-usco-v3:1}.

\begin{corollary}
  \label{corollary-min-usco-v10:1}
  Let $X$ be an extremally disconnected space, $Y$ be a space and
  $\varphi:X\sto Y$ be an usco mapping. Then $\varphi$ has a
  continuous selection. If moreover $\varphi$ is also perfect, then it
  has a perfect selection.
\end{corollary}

In the category of Hausdorff spaces and perfect maps, a space $E$ is
called \emph{projective} if for any perfect onto map $g:Y\to X$, every
perfect map $\pi:E\to X$ lifts to a perfect map $f:E\to Y$, i.e.\ for
which the following diagram is commutative:
\begin{center}
\begin{tikzcd}
  Y \arrow[d, swap, "g"] & \arrow[l, dashed, swap, "f"] \arrow[dl,
  "\pi"]E  \\
  X &
\end{tikzcd}  
\end{center}
In this setting, a projective space $E$ is called the \emph{absolute}
of a space $X$ if there exists an irreducible perfect map $\pi:E\to X$
and every irreducible perfect map $h:Y\to E$ of a space $Y$ is a
homeomorphism. Here, a continuous onto map $\pi:E\to X$ is
\emph{irreducible} if $\pi(A)\neq X$ for every proper closed subset
$A\subset E$.\medskip

Projective spaces go back to the fundamental work of Gleason
\cite{Gleason1958}. Briefly, it was shown in \cite[Theorem
1.2]{Gleason1958} that each projective space is extremally
disconnected and in \cite[Theorem 2.5]{Gleason1958} that in the
subcategory of compact spaces and continuous maps, the projective
spaces are precisely the extremally disconnected spaces. Moreover, it
was shown in \cite[Theorem 3.2]{Gleason1958} that each compact space
$X$ has an absolute. For compact projective spaces, the interested
reader may also consult Rainwater \cite{Rainwater1959} where a
simplified approach to some of Gleason's results is
presented. Furthermore, it was shown in \cite{Rainwater1959} that a
compact space is extremally disconnected if and only if it is a
retract of the \v{C}ech-Stone compactification of a discrete
space.\medskip

Gleason's results were extended to arbitrary (Hausdorff) spaces, there
is a long list of authors who have contributed in this regard. Here we
will mention only some of the most significant such results. In the
setting of paracompact spaces and perfect maps, the theory of
absolutes and irreducible perfect maps is given with all details in
Ponomarev's paper \cite{Ponomarev1962}. This theory was subsequently
extended to all regular Hausdorff spaces in the papers by Iliadis and
Ponomarev \cite{Iliadis1963a, Ponomarev1963a} and Flachsmeyer
\cite{flachsmeyer:63}. Some of these results were later rediscovered
by Strauss \cite{Strauss1967}. Regarding not necessarily regular
spaces, the following essentially well-known properties easily follow
from Proposition \ref{proposition-min-usco-v10:1} and Corollary
\ref{corollary-min-usco-v10:1}.

\begin{corollary}
  \label{corollary-min-usco-v11:1}
  Each extremally disconnected space $E$ is projective in the category
  of Hausdorff spaces and perfect maps. If, moreover, there exists an
  irreducible perfect map $\pi:E\to X$ onto some space $X$, then $E$
  is the absolute of $X$.
\end{corollary}

\begin{proof}
  Let $X$ and $Y$ be spaces, and $g:Y\to X$ and $\pi:E\to X$ be
  perfect maps. If $g$ is also onto, then the composition
  $\varphi=g^{-1}\circ \pi:E\sto Y$ is a perfect mapping. Hence, by
  Corollary \ref{corollary-min-usco-v10:1}, $\varphi$ has a perfect
  selection $f:E\to Y$. Accordingly, $E$ is projective. Finally, let
  us assume that $h:Y\to E$ is an irreducible perfect map. Then
  $\psi=h^{-1}:E\sto Y$ is a minimal perfect mapping and, according to
  Proposition~\ref{proposition-min-usco-v10:1}, it is
  singleton-valued. Therefore, $h$ is a homeomorphism.
\end{proof}

For a natural generalisation of Corollary
\ref{corollary-min-usco-v10:1} to arbitrary spaces, the interested
reader is referred to Ul'yanov's paper \cite{Ulyanov1977}. We conclude
this section with the following simple characterisation of extremally
disconnected spaces in terms of selections.

\begin{proposition}
  \label{proposition-min-usco-v7:1}
  For a space $X$, the following are equivalent\textup{:}
  \begin{enumerate}
  \item\label{item:min-usco-v14:1} $X$ is extremally disconnected.
  \item\label{item:min-usco-v14:2} Each usco mapping $\varphi:X\sto Y$
    has a continuous selection.
  \item\label{item:min-usco-v14:3} Each usco mapping
    $\varphi:X\sto\, \{0,1\}$ has a continuous selection.
  \end{enumerate}
\end{proposition}

\begin{proof}
  The implication
  \ref{item:min-usco-v14:1}$\implies$\ref{item:min-usco-v14:2} is
  Corollary \ref{corollary-min-usco-v10:1}, while
  \ref{item:min-usco-v14:2}$\implies$\ref{item:min-usco-v14:3} is
  trivial. Suppose that \ref{item:min-usco-v14:3} holds, and
  $U\subset X$ is an open set. Next, define an usco mapping
  $\varphi:X\sto\, \{0,1\}$ by $\varphi(x)=0$ if $x\in U$;
  $\varphi(x)=1$ if $x\notin \overline{U}$ and ${\varphi(x)=\{0,1\}}$
  otherwise. Then according to \ref{item:min-usco-v14:3}, $\varphi$
  has a continuous selection $f:X\to \{0,1\}$ and, therefore,
  $\overline{U}=f^{-1}(0)$ is a clopen set.
\end{proof}

\section{Compactifications and Extremal Disconnectedness} 
\label{sec:extens-select-extr}

It is a simple exercise that each dense subset of an extremally
disconnected space is also extremally disconnected, and that each
regular extremally disconnected space is completely regular.\medskip

In this section, we will show that the selection-extension property in
Theorem~\ref{theorem-min-usco-v4:1} follows from properties of the
\v{C}ech-Stone compactification of regular extremally disconnected
spaces. These properties are well known and are summarised below, the
interested reader is referred to 6M of \cite{gillman-jerison:60}.

\begin{proposition}
  \label{proposition-min-usco-v14:1}
  The \v{C}ech-Stone compactification of a regular extremally
  disconnected space is also extremally disconnected, and each compact
  extremally disconnected space is the \v{C}ech-Stone compactification
  of each of its dense subspaces.
\end{proposition}

Using Proposition \ref{proposition-min-usco-v14:1}, we first give a
very simple proof that Theorem \ref{theorem-min-usco-v4:1} is valid
when the range is an arbitrary completely regular space.

\begin{proposition}
  \label{proposition-min-usco-v16:1}
  Let $X$ be an extremally disconnected regular space, $Y$ be a
  completely regular space and $\varphi:X\sto Y$ be an usco
  mapping. Then for each dense subset $A\subset X$, each continuous
  selection for $\varphi\uhr A$ can be extended to a continuous
  selection for~$\varphi$.
\end{proposition}

\begin{proof}
  Let $A\subset X$ be dense and $g:A\to Y$ be a continuous selection
  for $\varphi\uhr A$. Since $\varphi$ is compact-valued and u.s.c.,
  it remains so as a set-valued mapping from $X$ to the \v{C}ech-Stone
  compactification $\beta Y$ of $Y$. Thus, ${g:A\to \beta Y}$ is a
  continuous selection for $\varphi\uhr A: A\sto \beta Y$. We can now
  apply Proposition~\ref{proposition-min-usco-v14:1} that
  $\beta A=\beta X$. Accordingly, $g$ can be extended to a continuous
  map ${\beta g:\beta X\to \beta Y}$. Finally, set $f=\beta g\uhr
  X$. Since the graph $\Gamma(\varphi)\subset X\times Y$ of $\varphi$
  is closed in $X\times \beta Y$ and
  $\Gamma(g)\subset \Gamma(\varphi)$, it follows that
  $\Gamma(f)=\overline{\Gamma(g)}\subset \Gamma(\varphi)$. This is
  equivalent to the fact that $f:X\to Y$ is a selection for~$\varphi$.
\end{proof}

Now, we also have the following application of projective spaces
showing that Theorem \ref{theorem-min-usco-v4:1} is valid when the
range is only assumed to be regular.

\begin{corollary}
  \label{corollary-min-usco-v16:1}
  Let $X$ be an extremally disconnected regular space, $Y$ be a
  regular space and $\varphi:X\sto Y$ be an usco mapping. Then for
  each dense subset $A\subset X$, each continuous selection for
  $\varphi\uhr A$ can be extended to a continuous selection
  for~$\varphi$. 
\end{corollary}

\begin{proof}
  Let $A\subset X$ be dense and $g:A\to Y$ be a continuous selection
  for $\varphi\uhr A$. As shown in \cite{Iliadis1963a,
    Ponomarev1963a}, see also \cite{flachsmeyer:63}, there exists an
  extremally disconnected regular space $E$ and a perfect onto map
  $\pi:E\to Y$. Then $\pi^{-1}:Y\sto E$ is usco, hence so is the
  composite mapping $\psi=\pi^{-1}\circ \varphi:X\sto E$. Similarly,
  $\pi^{-1}\circ g:A\sto E$ is also usco and it follows from Corollary
  \ref{corollary-min-usco-v10:1} that it has a continuous selection
  $\tilde{g}:A\to E$. Hence, by Proposition
  \ref{proposition-min-usco-v16:1}, $\tilde{g}$ can be extended to a
  continuous selection $h:X\to E$ for $\psi$ because $E$ is completely
  regular.  Accordingly, $f=\pi\circ h:X\to Y$ is a continuous
  selection for $\varphi$ with $f\uhr A=g$.
\end{proof}

Finally, let us remark that the selection-extension property in
Proposition~\ref{proposition-min-usco-v16:1} is equivalent to extremal
disconnectedness. This is a simple consequence of the following
essentially known properties.

\begin{proposition}
  \label{proposition-min-usco-vgg:1}
  For a regular space $X$, the following are equivalent\textup{:}
  \begin{enumerate}
  \item\label{item:min-usco-vgg:1} $X$ is extremally disconnected.
  \item \label{item:min-usco-vgg-4th:1} $X$ is completely regular and
    $\beta A=\beta X$ for every dense subset $A\subset X$.
  \item\label{item:min-usco-vgg:2} If $A\subset X$ is a dense subset
    and $Y$ is a compact space, then each continuous map $g:A\to Y$
    can be extended to a continuous map $f:X\to Y$.
  \end{enumerate}  
\end{proposition}

\begin{proof}
  The implications
  \ref{item:min-usco-vgg:1}$\implies
  $\ref{item:min-usco-vgg-4th:1}$\implies$\ref{item:min-usco-vgg:2}
  follow from Proposition \ref{proposition-min-usco-v14:1} and the
  property of the \v{C}ech-Stone compactification. Suppose that
  \ref{item:min-usco-vgg:2} holds, and take an open set $U\subset
  X$. Next, as in 6M of \cite{gillman-jerison:60}, let
  $A=U\cup \left(X\setminus \overline{U}\right)$ and $g:A\to \{0,1\}$
  be defined by $g(x)=0$ if $x\in U$ and $g(x)=1$ if
  $x\notin \overline{U}$. Then by \ref{item:min-usco-vgg:2}, $g$ can
  be extended to a continuous map $f:X\to \{0,1\}$. Accordingly,
  $\overline{U}=f^{-1}(0)$ is a clopen set and $X$ is extremally
  disconnected.
\end{proof}

\subsection*{Acknowledgement.} In conclusion, the author would like to
express his sincere gratitude to the referee for several valuable
remarks and suggestions.


\end{document}